\documentclass[12pt,a4paper]{article}

\usepackage{amsmath,amssymb, amsfonts, amsthm}

\newtheorem{theorem}{Theorem}[section]
\newtheorem{lemma}{Lemma}[section]

\newtheorem{corollary}[theorem]{Corollary}

\theoremstyle{definition}
\newtheorem{remark}[theorem]{Remark}
\newtheorem{comment}[theorem]{Comment}

\numberwithin{equation}{section}

\newcommand{\Z}{\mathbb Z}

\renewcommand{\:}{\colon}
\renewcommand{\>}{\rightarrow}

\title{Criteria of ergodicity for $p$-adic dynamical systems in terms of coordinate functions}

\author{Andrei Khrennikov\footnote{Corresponding author.} and Ekaterina Yurova\\ International Center for Mathematical Modelling\\
in Physics and Cognitive Sciences\\ 
Linnaeus University, V\"axj\"o, S-35195, Sweden\\
andrei.khrennikov@lnu.se}

\begin{document}
\maketitle

\begin{abstract}
This paper is devoted to the problem of ergodicity of $p$-adic dynamical systems. Our aim is to present criteria
of ergodicity in terms of coordinate functions corresponding to digits in the canonical expansion of $p$-adic
numbers. The coordinate representation can be useful, e.g., for applications to cryptography. Moreover, by using this representation we can consider non-smooth $p$-adic transformations. The basic technical tools
are van der Put series and usage of algebraic structure (permutations) induced by coordinate functions with partially frozen
variables. We illustrate the basic theorems by presenting concrete classes of ergodic functions. 
\end{abstract}

\section{Introduction}

Algebraic and arithmetic dynamics are actively developed fields of general theory of dynamical systems. The bibliography collected by Franco Vivaldi \cite{Vivaldi}contains 216 articles and books; extended bibliography also can be found in books of Silverman \cite{Silverman} and Anashin and Khrennikov \cite{ANKH}. Such studies are based on combination of number theory and theory of dynamical 
systems. And, as it often happens in mathematics, combination can induce novel constructions interesting for both areas of research.
By Ostrowsky theorem \cite{Schikhof} fields of the real and $p$-adic numbers are the most natural completions of the field of rational numbers ${\bf Q}.$ This is one of (purely mathematical) motivations to study dynamical systems in ${\bf Q}_p,$ see. e.g.,  
 \cite{Albeverio1}--\cite{Vivaldi2} (the complete list of reference would be very long; 
hence, we refer to \cite{Vivaldi}, \cite{ANKH}, \cite{Khrennikov and Nilsson}). We can also mention applications of 
$p$-adic dynamical system in physics, cognitive science and genetics \cite{KHRBOOK}, \cite{CONS} \cite{DUB},
\cite{Albeverio1}, \cite{Albeverio2}, \cite{Gen1}--\cite{Gen3} .

This paper is devoted to the problem of ergodicity of $p$-adic dynamical systems. Our aim is to present criteria
of ergodicity in the terms of the coordinate functions corresponding to the digits in the canonical expansion of $p$-adic
numbers. We remind that any $p$-adic integer (an element of the ring ${\bf Z}_p)$ can be expanded into the series:
\begin{equation}
\label{111}
x=x_0 + x_1 p + ...+ x_k p^k+..., \;\; x_j \in\{0,1,...,p-1\}.
\end{equation}
Hence, any map $f: {\bf Z}_p \to {\bf Z}_p$ can be represented in the form:
Any function $f\:\Z_p\>\Z_p$ can be represented as
\[
f(x)=\delta_0 (x)+ p\delta_1 (x)+ \ldots+ p^k \delta_k (x)+\ldots ,
\]
where functions $\delta_k (x), k=0,1,2,\ldots$ are $k$th digits in a base-$p$ expansion of the number $f(x),$ i.e. $\delta_k \: \Z_p\rightarrow \left\{0,1,\ldots,p-1\right\}.$
Our aim to find criteria of measure-preserving and ergodicity in terms of functions
$\delta_k.$ We restrict our study to the class of so called {\it compatible functions} or in terms of $p$-adic analysis 
1-Lipschitz functions, see Anashin \cite{Unif0}--\cite{Erg} and also the monograph \cite{ANKH}.  
The coordinate representation can be useful, e.g., for applications to cryptography. Moreover, by using this representation we can consider non-smooth $p$-adic transformations.

The basic technical tools are van der Put series and usage of 
algebraic structure (permutations) induced by coordinate functions with partially frozen variables. 
(In section \ref{TITI} we shall explain this technique in details.) 

The representation of the function by the {\it van der Put series}  is actively used in
$p$-adic analysis, see e.g. Mahler \cite{Mahler} and Schikhof \cite{Schikhof}. 
Marius van der Put introduced this series in his dissertation ``Algebres de fonctions continues p-adiques'' at Utrecht Universiteit in 1967, \cite{vdp}. There are numerous results in studies of functions with zero derivatives, antiderivation \cite{Schikhof} obtained using van der Put series. Later van der Put basis was adapted to the case of $n$-times continuously differentiable functions in one and several variables \cite{DeSmedt}. First results on applications of the van der Put series
in theory of $p$-adic dynamical systems, the problems of ergodicity and measure preserving,  were obtained in 
\cite{DAN}, see also \cite{Ya}, \cite{2},  \cite{Ya0}, \cite{MeraJNT}.  In this paper we apply this technique to find criteria of ergodicity in terms of the coordinate representation.

In some theorems (Proposition 3.1 and 3.4) conditions of ergodicity are formulated in terms of integral sums for the 
Volkenborn integral, see, e.g., Schikhof \cite{Schikhof}. One can expect to find formulations of ergodicity in terms of this
integral (playing an important role in number theory). However, at the moment this is an open problem. 

We illustrate the basic theorems by presenting concrete classes of ergodic functions which satisfy conditions of these theorems.

\section{$P$-adics}

Let $p>1$ be an arbitrary prime number. The ring of $p$-adic integers is denoted by  the symbol $\Z_p.$ The $p$-adic valuation
is denoted by $\vert \cdot\vert_p.$ We remind that this valuation satisfies the {\it strong triangle inequality}:
$$
\vert x+y\vert_p \leq \max [\vert x\vert_p, \vert y\vert_p].
$$
This is the main distinguishing property of the $p$-adic valuation inducing essential departure from the real or complex analysis
(and hence essential difference of  $p$-adic dynamical systems from  real and complex dynamical systems).

The van der Put series are defined in the following way. Let $f\:\Z_p\>\Z_p$ be a continuous function. Then there exists a unique sequence of $p$-adic coefficients
$B_0,B_1,B_2, \ldots$ such that 
\begin{eqnarray}
f(x)=\sum_{m=0}^{\infty}
B_m \chi(m,x) 
\end{eqnarray}
for all $x \in \Z_p.$
Here the \emph{characteristic function} $\chi(m,x)$ is given by
\begin{displaymath}
\chi(m,x)=\left\{ \begin{array}{cl}
1, \mbox{if}
\ \left|x-m \right|_p \leq p^{-n} \\
0,  \mbox{otherwise}
\end{array} \right.
\end{displaymath}
where $n=1$ if $m=0$, and $n$ is uniquely defined by the inequality 
$p^{n-1}\leq m \leq p^n-1$ otherwise (see Schikhof's book \cite{Schikhof} for detailed presentation 
of theory of van der Put series). 

 Let $f\:\Z_p\>\Z_p$ be a function  and
let $f$ satisfy the {\it Lipschitz condition with constant $1$} (with respect to the $p$-adic valuation $\left|\cdot \right|_p):$
 $$\left| f(x)- f(y)\right|_p \leq \left|x-y \right|_p$$
for all $x,y \in \Z_p.$ 

We state again that a mapping of an algebraic system $A$ to itself is called {\it compatible} if it preserves all the congruences of $A.$ 
It is easy to check that a map $f: \Z_p \to \Z_p$ is Lipschitz one iff it is compatible (with respect to $\rm{mod}\; p^k, k=1,2,...$
congruences).

By Proposition 3.35 [p. 63, \cite{ANKH}] the function $f\:\Z_p\>\Z_p$ is compatible iff each coordinate function $\delta_k, k=0,1,2,\ldots$ does not depend on 
$x_{k+1}, x_{k+2}, \ldots,$ where $x=x_0+px_1+p^2x_2+\ldots$ is canonical representation of $p$-adic number $x.$
Thus the compatible function $f\:\Z_p\>\Z_p$ has the following coordinate representation:
\begin{equation}
\label{coordf}
f(x_0+px_1+\ldots +p^kx_k+ \ldots)=\varphi_0(x_0)+p\varphi_1(x_0,x_1)+\ldots+p^k\varphi_k(x_0,x_1,\ldots,x_k)+\ldots\;,
\end{equation}
where $\varphi_k=\varphi_k(x_0,x_1,\ldots,x_k)$ are $p$-valued functions that depend on $p$-valued variables $x_0,x_1,\ldots,x_k,$ $k=0,1,2,\ldots .$

The space $\Z_p$ is equipped by the natural probability measure, namely, the Haar measure $\mu_p$ normalized
so that $\mu_p(\Z_p) = 1.$ 

Recall that a mapping $f\:\mathbb S\rightarrow \mathbb S$ of a measurable space $\mathbb S$ with a probability measure $\mu$ is called measure preserving if 
$\mu(f^{-1}(S))=\mu(S)$ for each measurable subset $S\subset \mathbb S.$

\section{General criteria of ergodicity for any prime $p.$}
\label{TITI}

\begin{theorem}
\label{measurecoord}
Let the $p$-adic compatible function $f\:\Z_p\>\Z_p$ has the coordinate representation (\ref{coordf}).
The function $f$ preserves measure if and only if
\begin{enumerate}
\item
$\varphi_0$ is bijective on the set of residues modulo $p;$
\item  $\varphi_k(\bar{x}_{k-1},x_k)$ is bijective on the set of residues modulo $p$ for any $k=1,2,\ldots$ and fixed values $\left(x_0,x_1,\ldots,x_{k-1}\right)=\bar{x}_{k-1}.$
\end{enumerate}
\end{theorem}

\begin{proof}
Let $f(x)=\sum_{m=0}^{\infty} B_m \chi(m,x)=\sum_{m=0}^{\infty} p^{\left\lfloor \log_p m \right\rfloor} b_m \chi(m,x)$ be the van der Put representation of the function $f$ and
$\bar{x}=x_0+px_1+\ldots +p^{n-1}x_{n-1},$ $\bar{\bar{x}}=\left(x_0,x_1,\ldots,x_{n-1}\right),$ $x_i \in \left\{0,1,\ldots,p-1\right\}.$
Let us find values of the coefficients $b_m,$ where $m=x_0+px_1+\ldots +p^{n-1}x_{n-1}+p^nx_n=\bar{x}+p^nx_n,$ $x_n\neq0.$
Then $b_{x_0}(f)=B_{x_0}(f)=f(x_0),$ $x_0\in \left\{0,1,\ldots,p-1\right\},$

and 
$$
b_{\bar{x}+p^nx_n}(f)=\frac{1}{p^n}B_{\bar{x}+p^nx_n}(f)= \frac{f(\bar{x}+p^nx_n)-f(\bar{x})}{p^n}
$$
$$
=\left(\varphi_n(\bar{\bar{x}},x_n)-\varphi_n(\bar{\bar{x}},0)\right) + p \left(\varphi_{n+1}(\bar{\bar{x}},x_n,0)-\varphi_{n+1}(\bar{\bar{x}},0,0)\right)+\ldots
$$

Thus 
\begin{align*}
&b_{x_0}(f)\equiv \varphi_0(x_0) \bmod p; \\
&b_{\bar{x}+p^nx_n}(f)\equiv \varphi_n(\bar{\bar{x}},x_n)-\varphi_n(\bar{\bar{x}},0) (\bmod p).
\end{align*}

By Theorem 2.1 \cite{MeraJNT} a function $f$ preserves measure if and only if
\begin{enumerate}
\item $b_0(f),b_1(f),\ldots, b_{p-1}(f)$ establish a complete set of residues modulo $p;$
\item $b_{\bar{x}+p^n}(f), \ldots ,b_{\bar{x}+p^n(p-1)}(f)$
 are all nonzero residues modulo $p.$
\end{enumerate}

Then the first condition is equivalent to bijectivity of $\varphi_0$ and the second condition to bijectivity of $\varphi_n(\bar{\bar{x}},h),$ $h\in \left\{0,1,\ldots,p-1\right\}.$

\end{proof}

Let us introduce new notations for compatible measure-preserving function $f\:\Z_p\>\Z_p,$ which has coordinate form (\ref{coordf}).
By Theorem \ref{measurecoord} for any $k=0,1,2,\ldots$ and any fixed $\bar{x}_{k-1}=(x_0,x_1,\ldots ,x_{k-1})$ the function $\varphi_k(\bar{x}_{k-1},x_k)$
is a bijective transformation of the set of residues modulo $p,$ i.e. defines permutation on the set $\left\{0,1,\ldots,p-1\right\}.$
Such permutation we denote as $\varphi_{k,\bar{x}_{k-1}}.$ Product of permutations denote as ``$\circ$'', i.e. 
$\varphi_{k,\bar{y}_{k-1}}\circ \varphi_{k,\bar{x}_{k-1}}=\varphi_{k}\left(\bar{y}_{k-1},\varphi_{k}\left(\bar{x}_{k-1},x_k\right)\right),$
where $x_k\in \left\{0,1,\ldots,p-1\right\}.$

The function $f$ induces a transformation of the residue ring modulo $p^k,$ i.e. $f_k\: x \bmod p^{k+1}\longmapsto f(x) \bmod p^{k+1}$ 
(correctness of this map follows from compatibility of $f$).
Symbol $f_k$ we also will use to denote coordinate form of its representation, i.e.
$f_k\:(x_0,x_1,\ldots,x_k)\longmapsto \left(\varphi_0(x_0),\varphi_1(x_0,x_1),\ldots, \varphi_k(x_0,x_1,\ldots,x_k)\right).$

Now let us state a general criterion of ergodicity for a compatible function in terms of the coordinate functions.

\begin{theorem}
\label{generg}
Let the $p$-adic compatible function $f\:\Z_p\>\Z_p$ be presented in the coordinate form (\ref{coordf}),
where $\varphi_0$ and $\varphi_{k,\bar{x}_{k-1}},
k=1,2,\ldots,$ be permutations on the set of residues modulo $p.$
The function $f$ is ergodic if and only if
\begin{enumerate}
\item the map $\varphi_0$ is transitive on the set of residues modulo $p;$
\item the permutation $$F_{k,\bar{x}_{k-1}}=\varphi_{k,f_{k-1}^{(p^k-1)}(\bar{x}_{k-1})} \circ \varphi_{k,f_{k-1}^{(p^k-2)}(\bar{x}_{k-1})} \circ\ldots\circ \varphi_{k,\bar{x}_{k-1}}$$
is transitive on the set of residues modulo $p$ for any $k=1,2,\ldots ,$ where $f_k^{(s)}\left(\bar{x}_{k-1}\right)=\underbrace{f_k(f_k(\ldots f_k(\bar{x}_{k-1}))\ldots )}_{s},$
$f_k=f \bmod p^{k+1}$ and $f_0=\varphi_0.$
\end{enumerate}
\end{theorem}

\begin{proof}
Let us find the coordinate representation for $f_k^{(p^k)}(\bar{x}_k),$ $k=1,2,\ldots .$
Note that from compatibility of $f$ it  follows that $f_k \bmod p^k=f_{k-1}.$ Then we have:
$$f_k^{(p^k)}=\left(F_0, F_{1,\bar{x}_0},\ldots, F_{{k-1},\bar{x}_{k-2}}, F_{k,\bar{x}_{k-1}} \right),$$
where 
\begin{align}
\label{coordpr}
&F_0=\underbrace{\varphi_0 \circ \varphi_0 \circ \ldots  \circ \varphi_0}_{p^{k-1}}; \\
&F_{1,\bar{x}_0}= \varphi_{1,f_1^{(p-1)}(\bar{x}_0)} \circ \varphi_{1,f_1^{(p-2)}(\bar{x}_0)} \circ \ldots \circ \varphi_{1,\bar{x}_0}; \nonumber\\
&\ldots \nonumber\\
&F_{{k-1},\bar{x}_{k-2}}= \varphi_{{k-1},f_{k-2}^{(p^{k-1}-1)}(\bar{x}_{k-2})} \circ \varphi_{{k-1},f_{k-2}^{(p^{k-1}-2)}(\bar{x}_{k-2})} \circ \ldots \circ \varphi_{{k-1},\bar{x}_{k-2}}; \nonumber\\
&F_{k,\bar{x}_{k-1}}= \varphi_{k,f_{k-1}^{(p^k-1)}(\bar{x}_{k-1})} \circ \varphi_{k,f_{k-1}^{(p^k-2)}(\bar{x}_{k-1})} \circ \ldots \circ \varphi_{k,\bar{x}_{k-1}}; \nonumber
\end{align}

\medskip

We start from the proof that conditions (1), (2) are necessary. Assume that $f$ is ergodic.   

By Theorem 4.23 (p.99, \cite{ANKH}) if the function $f$ is ergodic, then, for any $s=0,1,2,\ldots,$ the
function $f_s \: x\bmod p^{s+1}\longmapsto f(x)\bmod p^{s+1}$ is transitive on the set of residues modulo $p^{s+1}$ and, in fact,
$p^{s+1}$ is the minimal integer such that $f_s^{(p^{s+1})} (\bar{x}_{s})=\bar{x}_{s}.$
It means that the permutations $F_0,F_{1,\bar{x}_0},\ldots, F_{{k-1},\bar{x}_{k-2}}$ from the 
coordinate representation of the map $f_k^{(p^k)}$ {\it are identical.} Indeed, in this case each permutation is a degree of the identical permutation $\epsilon.$
As we know,   $f_k^{(p^{k+1})} (\bar{x}_{k})=\bar{x}_{k}$  and  $p^{k+1}$ is minimal.
Then from the representation
\[
f_k^{(p^{k+1})}=f_k^{(p^k)^{(p)}}=\left(\epsilon,\epsilon,\ldots,\epsilon,\underbrace{F_{k,\bar{x}_{k-1}} \circ \ldots \circ F_{k,\bar{x}_{k-1}}}_{p}\right)
\]
it follows that 
\begin{equation}
\label{HH77}
\underbrace{F_{k,\bar{x}_{k-1}} \circ \ldots \circ F_{k,\bar{x}_{k-1}}}_{p}=\epsilon,
\end{equation}
and, moreover,   $p$ is minimal natural number for holding the condition (\ref{HH77}). Hence, the 
 permutation $F_{k,\bar{x}_{k-1}}$ is transitive.
Since $f_0=\varphi_0,$ then transitivity of the map $\varphi_0$ follows from Theorem 4.23 (p.99, \cite{ANKH}).

\medskip
We now present the proof that conditions (1), (2) are sufficient for ergodicity.
We shall use induction with respect to the parameter $k=0,1,\ldots,$ to show 
transitivity of the functions $f_k \: x\bmod p^{k+1}\longmapsto f(x)\bmod p^{k+1}.$

For $k=0$ transitivity of $f_0$ follows from the relation $f_0=\varphi_0$ and transitivity of $\varphi_0$ follows from the first condition of Theorem.

Let $f_{k-1}$ be transitive on the set of residues modulo $p^k.$ Since $f_k \bmod p^k=f_{k-1},$ then $f_k^{(p^k)}=\left(\epsilon,\epsilon,\ldots,\epsilon,F_{k,\bar{x}_{k-1}}\right),$ where $F_{k,\bar{x}_{k-1}}$ is defined in (\ref{coordpr}) and $\epsilon$ is the identity permutation on the set of residues modulo $p.$
By condition (2) of Theorem, the map $F_{k,\bar{x}_{k-1}}$ is the transitive permutation on the set of residues modulo $p.$ 
Therefore,
\[
f_k^{(p^{k+1})}=f_k^{(p^k)^{(p)}}=\left(\epsilon,\epsilon,\ldots,\epsilon,\underbrace{F_{k,\bar{x}_{k-1}} \circ \ldots \circ F_{k,\bar{x}_{k-1}}}_{p}\right)=\left(\epsilon,\epsilon,\ldots,\epsilon\right),
\]
and such $p^{k+1}$ is minimal. It means that $f_k$ is transitive on the set of residues modulo $p^{k+1}.$
Thus we proved that the functions $f_k=f(x)\bmod p^{k+1},$ $k=0,1,\ldots$ are transitive. Then by Proposition 4.35 (p.105, \cite{ANKH}) the function $f$ is ergodic.
\end{proof}

\begin{comment}{\small 
To check a function on ergodicity by using conditions of Theorem \ref{generg}, one should check transitivity of the permutations $F_{k,\bar{x}_{k-1}},$
$k=1,2,\ldots.$ Each permutation is a product of permutations $\varphi_{k,\bar{x}_{k-1}},$ $\bar{x}_{k-1} \in \left\{0,\ldots,p^{k-1}-1\right\}.$ The order of their appearance in the resulting product (i.e. in $F_{k,\bar{x}_{k-1}}$) is 
defined by the sequence of residues modulo $p^k$
$$\tau_k=\left\{\bar{x}_{k-1}, f_{k-1}(\bar{x}_{k-1}),\ldots, f_{k-1}^{(p^k-2)}(\bar{x}_{k-1}),f_{k-1}^{(p^k-1)}(\bar{x}_{k-1})\right\}.$$

In other words, to check transitivity of the function $f_k=f \bmod p^{k+1}$ one should construct the sequence $\tau_k,$
find $F_{k,\bar{x}_{k-1}}$ and verify its transitivity.

Note that the order of  residues modulo $p^k$ in the sequence $\tau_k$ is significantly important.
This is due to the fact that the symmetric group $S_p$ (permutations on $\Z/p\Z$) is nonabelian. Therefore, in general
by determining $F_{k,\bar{x}_{k-1}}$ we cannot change the order of  
the permutations $\varphi_{k,\bar{x}_{k-1}}.$ 

Of course, in some special cases the permutations $\varphi_{k,\bar{x}_{k-1}}$ can commute. Then one can expect 
finding compact conditions of ergodicity for the corresponding class of functions. }
\end{comment}

\begin{remark}
Theorem \ref{generg} can be considered as generalization to the case  $p\not=2$ of 
Theorem 4.39 (Folklore) from the book \cite{ANKH}. 
The latter was ``known'' by people working in cryptography, but 
without formal 2-adic presentation. (And the authors of \cite{ANKH}
do not know any rigorous mathematical proof before the one presented in this book.) 
\end{remark}

\begin{remark}
\label{permvibor}
In fact, by checking the conditions of ergodicity of  Theorem \ref{generg}, one can choose the value of the parameter $\bar{x}_{k-1}=(x_0,x_1,\ldots,x_{k-1})$ in an arbitrary way.

Indeed, suppose that the second condition of Theorem \ref{generg} has been checked for $s=2,\ldots,k-2$ (for $s=1$ the second condition coincide with first). And it has been shown that the permutations $F_{s,\bar{x}_{s-1}}$ are transitive on the set of residues modulo $p^{s+1}.$ Then the functions $f_s$ are transitive modulo $p^{s+1}$ (otherwise the function $f$ is already non-ergodic and it is not necessarily to check the second condition).

In the second condition one has to check transitivity of the permutation
$$F_{k,\bar{x}_{k-1}}=\varphi_{k,f_{k-1}^{(p^{k-1}-1)}(\bar{x}_{k-1})} \circ \ldots\circ \varphi_{k,\bar{x}_{k-1}}$$
for some $\bar{x}_{k-1}=(x_0,x_1,\ldots,x_{k-1}).$

Let $\bar{y}_{k-1}=(y_0,y_1,\ldots,y_{k-1})\neq (x_0,x_1,\ldots,x_{k-1})=\bar{x}_{k-1}.$ We show that the cyclic structure of the permutations  $F_{k,\bar{x}_{k-1}}$ and $F_{k,\bar{y}_{k-1}}$ coincide.
By our assumption the function $f_{k-1}$ is transitive modulo $p^k.$ Then there exist an integer $r\leq p^{k-1}-1$ such that 
$f_{k-1}^{(r)}(\bar{x}_{k-1})=\bar{y}_{k-1}$ and, therefore, $f_{k-1}^{(t)}(\bar{y}_{k-1})=f_{k-1}^{(t+r)}(\bar{x}_{k-1}).$
Then
\begin{eqnarray}
\label{perm}
F_{k,\bar{x}_{k-1}}=\varphi_{k,f_{k-1}^{(p^{k-1}-1)}(\bar{x}_{k-1})} \circ \ldots \circ \varphi_{k,f_{k-1}^{(r+1)}(\bar{x}_{k-1})} \circ \varphi_{k,f_{k-1}^{(r)}(\bar{x}_{k-1})} \circ\ldots \nonumber\\
\ldots\circ \varphi_{k,\bar{x}_{k-1}}= \nonumber\\
=\varphi_{k,f_{k-1}^{(p^{k-1}-1-r)}(\bar{y}_{k-1})} \circ \ldots \circ \varphi_{k,f_{k-1}^{(2)}(\bar{y}_{k-1})} \circ \varphi_{k,f_{k-1}(\bar{y}_{k-1})} \circ \varphi_{k,f_{k-1}^{(r-1)}(\bar{x}_{k-1})} \circ\ldots \nonumber\\
\ldots\circ \varphi_{k,\bar{x}_{k-1}}.\nonumber\\
\end{eqnarray}

Consistently making conjugation of the permutation (\ref{perm}) by permutations \\
$\varphi_{k,\bar{x}_{k-1}},\ldots, \varphi_{k,f_{k-1}^{(r-1)}(\bar{x}_{k-1})},$ we obtain

\begin{align*}
&\varphi_{k,f_{k-1}^{(r-1)}(\bar{x}_{k-1})} \circ \ldots \circ \varphi_{k,\bar{x}_{k-1}} \circ F_{k,\bar{x}_{k-1}} 
\circ (\varphi_{k,\bar{x}_{k-1}})^{-1} \circ \ldots \circ (\varphi_{k,f_{k-1}^{(r-1)}(\bar{x}_{k-1})} )^{-1}= \\
&=\varphi_{k,f_{k-1}^{(r-1)}(\bar{x}_{k-1})}\circ \ldots \circ \varphi_{k,\bar{x}_{k-1}} \circ \varphi_{k,f_{k-1}^{(p^{k-1}-1-r)}(\bar{y}_{k-1})}\circ \ldots\circ \varphi_{k,f_{k-1}(\bar{y}_{k-1})}= \\
&=\varphi_{k,f_{k-1}^{(p^k-1)}(\bar{y}_{k-1})} \circ \ldots \circ \varphi_{k,f_{k-1}^{(p^k-r)}(\bar{y}_{k-1})} \circ 
\varphi_{k,f_{k-1}^{(p^k-1-r)}(\bar{y}_{k-1})}\circ \ldots \\
&\ldots\circ \varphi_{k,f_{k-1}(\bar{y}_{k-1})}= F_{k,\bar{y}_{k-1}}.
\end{align*}

Let $G=\varphi_{k,f_{k-1}^{(r-1)}(\bar{x}_{k-1})} \circ \ldots \circ \varphi_{k,\bar{x}_{k-1}}.$
Then $G \circ F_{k,\bar{x}_{k-1}} \circ G^{-1}=F_{k,\bar{y}_{k-1}},$
i.e., the permutations $F_{k,\bar{x}_{k-1}},$ $F_{k,\bar{y}_{k-1}}$ are conjugate permutations.
It is well known  that cyclic structure of conjugate permutations coincide.
It means that the property of transitivity for permutation $F_{k,\bar{x}_{k-1}}$ does not depend on the choice of $\bar{x}_{k-1}.$
\end{remark}

Taking into account Remark \ref{permvibor} and Theorem \ref{generg}, we state the following theorem.

\begin{theorem}
\label{generg2}
Let the $p$-adic compatible function $f\:\Z_p\>\Z_p$ be presented in the coordinate form (\ref{coordf}),
where $\varphi_0$ and $\varphi_{k,\bar{x}_{k-1}},
k=1,2,\ldots,$ are permutations on the set of residues modulo $p.$
The function $f$ is ergodic if and only if
\begin{enumerate}
\item  the map $\varphi_0$ is transitive on the set of residues modulo $p;$
\item the permutations $$F_{k,0}=\varphi_{k,f_{k-1}^{(p^k-1)}(0)} \circ \varphi_{k,f_{k-1}^{(p^k-2)}(0)} \circ\ldots\circ \varphi_{k,0}$$
are transitive on the set of residues modulo $p,$ for $k=1,2,\ldots ,$ where $f_k^{(s)}\left(\bar{x}_{k-1}\right)=\underbrace{f_k(f_k(\ldots f_k(\bar{x}_{k-1}))\ldots )}_{s},$
$f_k\equiv f \bmod p^{k+1}$ and $f_0=\varphi_0.$
\end{enumerate}
\end{theorem}

To prove this Theorem it is enough to set $\bar{x}_{k-1}=0$ and to use the result of the Remark \ref{permvibor}.

\begin{comment} {\small
In fact, Theorem \ref{generg2} states that if one checks ergodicity of a function using the conditions of 
Theorem \ref{generg}, then one can use an arbitrary value from the set $\Z/p^{k-1}\Z$ as a starting point for constructing a sequence
$f_{k-1}^{(s)},$ $s\in \left\{0,1,\ldots,p^{k-1}-1\right\}.$ For example, $\bar{x}_{k-1}=0.$}
\end{comment}

\begin{corollary}
\label{crerg}  
Let the $p$-adic compatible function $f\:\Z_p\>\Z_p$ be presented in the coordinate form (\ref{coordf}),
where the subfuctions of the coordinate functions, $\varphi_0$ and $\varphi_{k,\bar{x}_{k-1}},
k=1,2,\ldots,$ are permutations on the set $\Z/p\Z$ and $\phi_0$ is transitive.
Then it is always possible to construct 
permutations $g_1,g_2,\ldots,g_k,\ldots$ such that by setting $\tilde{\varphi}_{k,0}=g_k, k=1,2,\ldots,$
the corresponding function $\tilde{f}$ which is defined with the aid of subfunctions 
$\varphi_{k,\bar{x}_{k-1}},$ $\bar{x}_{k-1}\neq 0,$ and  $\tilde{\varphi}_{k,0}$
is ergodic.
\end{corollary}

\begin{proof}
Let us find permutations $g_1,g_2,\ldots,g_k,\ldots$ by induction.
By conditions of Theorem, permutation $\varphi_0=f_0\equiv f \bmod p$ is transitive on $\Z/p\Z,$ and, in particular, $\left\{f_0^{(p-1)}(0),\ldots,f_0(0),0\right\}=\left\{0,1,\ldots,p-1\right\}.$
By using notations from Theorem \ref{generg2}, we set $G_1=\varphi_{1,f_0^{(p-1)}(0)} \circ \varphi_{1,f_0^{(p-2)}(0)} \circ\ldots\circ \varphi_{1,f_0(0)}.$
We also choose some transitive permutation $H_1$ on $\Z/p\Z.$

Let $g_1=G_1^{-1} \circ H_1.$ Then set $\tilde{\varphi}_{1,0}=g_1,$ and we obtain that permutation $F_{1,0}=\varphi_{1,f_0^{(p-1)}(0)} \circ \varphi_{1,f_0^{(p-2)}(0)} \circ\ldots\circ \varphi_{1,f_0(0)}\circ \tilde{\varphi}_{1,0}=H_1$ is transitive on $\Z/p\Z.$ Moreover, the function $\tilde{f}_1\equiv \tilde{f} \bmod p^2$ is transitive on $\Z/p^2\Z.$
Now let permutations $g_1,g_2,\ldots,g_{k-1},\ldots$ have already been constructed, and the function $\tilde{f}_{k-1}\equiv \tilde{f} \bmod p^k$ is transitive on $\Z/p^k\Z.$

Let us construct a suitable permutation $g_k.$
Let $G_k=\varphi_{k,\tilde{f}_{k-1}^{(p^k-1)}(0)} \circ\ldots\circ \varphi_{k,\tilde{f}_{k-1}(0)}$ and $H_k$ be some transitive permutation on $\Z/p\Z.$
And set $g_k=G_k^{-1} \circ H_k.$ Then, for $\tilde{\varphi}_{k,0}=g_k,$ we obtain that the permutation $F_{k,0}=\varphi_{k,\tilde{f}_{k-1}^{(p^k-1)}(0)}  \circ\ldots\circ \varphi_{k,\tilde{f}_{k-1}(0)}\circ \tilde{\varphi}_{k,0}=H_k$ is transitive on $\Z/p\Z.$ And the function $\tilde{f}_k\equiv \tilde{f} \bmod p^{k+1}$ is transitive on $\Z/p^{k+1}\Z.$

Thus we have shown existence of permutations $g_1,g_2,\ldots,g_k,\ldots$ such that for $\tilde{\varphi}_{k,0}=g_k$ $k=1,2,\ldots$ permutations $F_{k,0}$ are transitive on $\Z/p\Z.$
By conditions of Theorem, $\varphi_0$ is transitive on $\Z/p\Z,$ therefore, the function $f$ is ergodic -- by Theorem \ref{generg2}.
\end{proof}

\begin{comment}
In the Corollary \ref{crerg}, in particular, we answer the  following question. Let $f$ be a measure-preserving compatible function. How much should one change such function to  get an ergodic function? It turns out that such change, with a certain degree of conditionality, is surprisingly small. 

Let us use the coordinate representation (\ref{coordf}) for our function $f.$
We consider the coordinate function $\varphi_k(x_0,x_1,\ldots,x_k);$ set $\varphi_{k,\bar{x}_{k-1}},$ a subfunction of the 
coordinate function $\varphi_k.$ Such a subfunction is obtained by fixing the values of the variables $x_0,\ldots,x_{k-1}.$ (Denote the vector $(x_0,\ldots,x_{k-1})$ with fixed values of their coordinates as $\bar{x}_{k-1}.$
The subfunction $\varphi_{k,\bar{x}_{k-1}}$ depends on values of the variable 
$x_k\in \left\{0,\ldots,p-1\right\}).$

All $\varphi_{k,\bar{x}_{k-1}}$ are permutations on $\Z/p\Z$ because the function $f$ preserves measure. 
Then, for the function $f,$ to be ergodic only requires to choose only one series of permutations in a special way,  for example, $\varphi_{k, 0}, k=1,2,....$
And other permutations may be arbitrary. The choice of the permutations $\varphi_{k,\bar{0}}, k=1,2,...,$ imposes the following restrictions:

Let $G_k=\varphi_{k,\tilde{f}_{k-1}^{(p^k-1)}(0)} \circ\ldots\circ \varphi_{k,\tilde{f}_{k-1}(0)},$ where the order of the permutations  in this product is significant. We can assume that this permutation  is an arbitrary on $\Z/p\Z$ (because $\varphi_{k,\bar{x}_{k-1}},$ $\bar{x}_{k-1}\neq 0$ can be set arbitrary).
Then we choose $\tilde{\varphi}_{k,0}$ in such a way that $G_k\circ \tilde{\varphi}_{k,0}$ be a transitive permutation on $\Z/p\Z.$

As $p$ is a prime number, then one can write this condition analytically. Namely, $\varphi_{k, 0}$ is a solution in permutations $\left(G_k\circ X\right)^p=e,$
where $e$ is the identical permutation.
This means that $\tilde{\varphi}_{k,0}$ can be chosen in $(p-1)!$ variants, i.e. proportion of suitable permutations is $\frac{1}{p}$ of all permutations on $\Z/p\Z.$
\end{comment}

\begin{corollary} 
\label{experm}
Let the $p$-adic compatible function $f\:\Z_p\>\Z_p$ be presented in the coordinate form (\ref{coordf}),
where the subfunctions (of the coordinate functions) $\varphi_0, \varphi_{k,\bar{x}_{k-1}},
k=1,2,\ldots$ are permutations on $\Z/p\Z.$  Suppose that the subfunctions have the form:  
\begin{equation}
\label{JJ7}
\varphi_{k,\bar{x}_{k-1}}=g_k^{n(\bar{x}_{k-1})},
\end{equation}
 where $g_k$ is a permutation on $\Z/p\Z$ and $n(\bar{x}_{k-1})$ is a positive integer ($g_k^0$- identity permutation), 
$k=1,2,\ldots .$ Then the function $f$ is ergodic if and only if
\begin{enumerate}
\item $\varphi_0,$ $g_k$ are transitive permutations;
\item $\sum_{\bar{x}_{k-1}=0}^{p^k-1} n(\bar{x}_{k-1})\neq 0 \bmod p,$ $k=1,2,\ldots .$
\end{enumerate}
\end{corollary}

\begin{proof}
Let $f$ be an ergodic function. By Theorem \ref{generg2} permutations $\varphi_0$ and $F_{k,0}=\varphi_{k,f_{k-1}^{(p^k-1)}(0)}  \circ\ldots\circ \varphi_{k,f_{k-1}(0)}\circ \varphi_{k,0},$ $k=1,2,\ldots $ are  transitive on $\Z/p\Z.$
Then 
\begin{align*}
&F_{k,0}=\varphi_{k,f_{k-1}^{(p^k-1)}(0)}  \circ\ldots\circ \varphi_{k,f_{k-1}(0)}\circ \varphi_{k,0}= \\
&=g_k^{n(f_{k-1}^{p^k-1} (0))}\circ\ldots \circ g_k^{n(0)}= \\
&=g_k^{n(p^k-1)+\ldots+n(0)}=g_k^{S_k},
\end{align*}
where $S_k=n(p^k-1)+\ldots+n(0).$

As $F_{k,0}$ is transitive on $\Z/p\Z,$ then $ord\; F_{k,0}=p.$

Let $ord\; g_k=\alpha.$ Then $p= ord\; F_{k,0}=ord\; g_k^{S_k}=\frac{\alpha}{\gcd (\alpha,S_k)},$ i.e. $p$ divides $\alpha.$
If the permutation $g_k$ has $r_1$ cycles of the length $1,$ $r_2$ cycles of the  length $2, \ldots r_p$ cycles of the length $p,$ and $r_1+2r_2+\ldots+pr_p=p,$
then $\alpha=ord\; g_k=lcm (i|r_i\neq 0).$
As  $p$ is a prime number and $p$ devides $\alpha,$  then the permutation $g_k$ consists of one cycle of length $p,$ i.e. $\alpha=p$  and $g_k$ is transitive permutation.
From $p= \frac{\alpha}{\gcd (\alpha,S_k)},$ $\alpha=p$ follows that $\gcd (\alpha,S_k)=\gcd (p,S_k)=1,$ i.e.
$S_k=\sum_{\bar{x}_{k-1}=0}^{p^k-1} n(\bar{x}_{k-1})\neq 0 \bmod p.$ 

And vice versa, let us find an order of each permutation $F_{k,0},$ $k=1,2,\ldots .$
By one of the conditions of this corollary   $\varphi_0=f_0=f \bmod p$ is a transitive permutation and, in particular, $\left\{f_0^{(p-1)}(0),\ldots,f_0(0),0\right\}=\left\{0,1,\ldots,p-1\right\}.$
Then $F_{1,0}=\varphi_{1,f_0^{(p-1)}(0)} \circ\ldots\circ \varphi_{1,f_0(0)}\circ \varphi_{1,0}=g_1^{n(p-1)+\ldots+n(0)}=g_1^{S_1},$
where $S_1=n(p-1)+\ldots+n(0).$
By initial conditions $g_1$ is transitive on $\Z/p\Z,$ then $ord\; g_1=p.$ As $S_1\neq 0 \bmod p,$ then $\gcd (p,S_k)=1$ and 
$ord\; F_{1,0}=ord\; g_1^{S_1}=\frac{p}{\gcd (p,S_k)}=p.$ Thus $F_{1,0}$ is transitive permutation on $\Z/p\Z.$
Moreover, the function $f_1=f \bmod p^2$ is transitive on $\Z/p^2\Z.$

Computation of the values $ord\; F_{k,0},$ $k=2,3\ldots $ is performed in a similar way, taking into account that at each step we prove transitivity of the functions
$f_{k-1}\equiv f \bmod p^k.$
Finally, we obtain that the permutations $F_{k,0},$ $k=1,2,\ldots $ are  transitive on $\Z/p\Z.$ As $\varphi_0$ is transitive on $\Z/p\Z,$
then the function $f$ is ergodic by Theorem \ref{generg2}.
\end{proof}

\begin{comment}
\label{cicl}
In the first Comment  was stated that by Theorem \ref{generg}  
to check a function on ergodicity one should check transitivity of the permutation $F_{k,\bar{x}_{k-1}},$
$k=1,2,\ldots.$ And each such permutation is a product of permutations $\varphi_{k,\bar{x}_{k-1}}\in S_p,$ $\bar{x}_{k-1} \in \left\{0,\ldots,p^{k-1}-1\right\},$ where $S_p$ is the symmetric group of permutations of the order $p.$ The order of their appearance in the resultant product is defined by the sequence of residues modulo $p^k$
$$\tau_k=\left\{\bar{x}_{k-1}, f_{k-1}(\bar{x}_{k-1}),\ldots, f_{k-1}^{(p^{k-1}-2)}(\bar{x}_{k-1}),f_{k-1}^{(p^k-1)}(\bar{x}_{k-1})\right\}.$$

If we suppose that the permutations $\varphi_{k,\bar{x}_{k-1}}\in S_p, \bar{x}_{k-1}\in \{0,1,..., p^{k-1}-1\},$ 
commute, then the conditions of ergodicity of $f$ are simplified.
In the Corollary \ref{experm}  we considered the case such that,  for each $k = 1,2, \ldots,$ the permutations $\varphi_{k,\bar{x}_{k-1}}\in S_p$
belong to the cyclic group generated by a permutation $g_k.$ In this case, all $\varphi_{k,\bar{x}_{k-1}}\in S_p$ commute, 
$F_{k,\bar{x}_{k-1}}$ does not depend on the order of elements in the sequence $\tau_k,$ and, moreover, $F_{k,\bar{x}_{k-1}}=g_k^\alpha$ for some $\alpha$
depending only on $k.$ Therefore, to verify the transitivity of $F_{k,\bar{x}_{k-1}},$ 
there is no need to build the sequence $\tau_k.$
This simplifies essentially the verification of ergodicity of $f.$
\end{comment}

\section{Examples of classes of ergodic functions for $p\neq 2$}
\begin{theorem}
\label{poln}
Let the $p$-adic ($p\neq 2$) compatible function $f\:\Z_p\>\Z_p$ be presented in the coordinate form (\ref{coordf}),
where the coordinate functions have the form:
\begin{equation}
\label{FORM} 
\varphi_{k, \bar{x}_{k-1}}(x_k)=x_k+\alpha(\bar{x}_k),
 \end{equation}
$ k=1,2,\ldots.$
The function $f$ is ergodic if and only if
\begin{enumerate}
\item $\varphi_0$ is a transitive  (monocycle) permutation on the set  of residues modulo $p;$
\item $2^{p-2}+\frac{1}{p^k}   \sum_{i=0}^{p^k-1} f(i)\neq 0 \bmod p,$ $k=2,3,\ldots .$
\end{enumerate}
\end{theorem}

\begin{proof}
Note that the function $f$ preserves measure.  Indeed, the coordinate function $\varphi_0$ is bijective by the assumption, and  the coordinate functions 
$\varphi_{k, \bar{x}_{k-1}}(x_k)$ are 
bijective as linear polynomials over the field of residues modulo $p.$ Then by the criterion of measure-preserving functions in terms of coordinate functions, see Theorem 2.1 in \cite{MeraJNT}, the function $f$ is measure-preserving.

For fixed $\bar{x}\in \left\{0,1,\ldots,p^k-1\right\}$ the function $x_k \to \varphi_{k, \bar{x}_{k-1}}(x_k)$ is a permutation on $\Z/p\Z, k=1,2,\ldots .$
As in Theorem \ref{generg2}, let us consider the permutations 
$$
F_{k,0}=\varphi_{k, f_{k-1}^{(p^k-1)}(0)} \circ \varphi_{k, f_{k-1}^{(p^k-2)}(0)} 
\circ\ldots\circ \varphi_{k,0}(x_k),
$$
where $k=1,2,\ldots .$
Here $f_{k-1}^{(s)}(0)= f_{k-1}(f_{k-1}(\ldots f_{k-1}(0))\ldots),$
the iteration of the $s$th order, and  $f_{k-1}\: x \bmod p^k \longmapsto f(x) \bmod p^k.$ (To be 
completely formal, instead of  $0,$  we have to use  the symbol $\bar{0}=\left(0,0,\ldots,0\right).$ However, to make notation simpler we shall 
proceed with  $0).$
By the assumption the permutations $\varphi_{k, \bar{x}_{k-1}}$ are given by linear polynomials of the 
form $\varphi_{k, \bar{x}_{k-1}}(x_k)= x_k+\alpha_k(\bar{x}).$ Therefore
\begin{align*}
&F_{k,0}=x_k+ \alpha_k(f_{k-1}^{(p^k)}(0)) + \alpha_k (f_{k-1}^{(p^k-2)}(0)) +\ldots+ \alpha_k(0)= \\
&=x_k+ \sum_{i=0}^{p^k-1} \alpha_k(f_{k-1}^{(i)}(0)).
\end{align*}

Let us show  by induction with respect to  $k=0,1,2,\ldots .$ 
that the functions $f_k\: \Z/p^k\longmapsto \Z/p^k,$ $f_k(x)\equiv f(x) \bmod p^{k+1}$ are transitive on $\Z/p^k.$
For $k=0,$ the function $f_0$ is transitive because $f_0=\varphi_0,$ where $\varphi_0$ is transitive on $\Z/p^k$ by the 
assumption.  Suppose now that $f_{k-1}$ is transitive on $\Z/p^k.$ Then
$$F_{k,0}=x_k+ \sum_{i=0}^{p^k-1} \alpha_k(f_{k-1}^{(i)}(0))=x_k+ \sum_{i=0}^{p^k-1} \alpha_k (i).$$
On the other hand, set $x=x_0+px_1+\ldots+p^kx_k=\bar{x}+p^kx_k,$ then we obtain
\begin{align*}
\sum_{x=0}^{p^k-1} f(x) \equiv \sum_{\bar{x}=0}^{p^k-1} \sum_{i=0}^k p^i \varphi_i (x_0,x_1,\ldots, x_i) \equiv \\
\equiv \sum_{\bar{x}=0}^{p^k-1} \left(p^k \alpha_k(\bar{x}) + \sum_{i=0}^{k-1} p^i \varphi_i (x_0,x_1,\ldots, x_i)  \right)\equiv \\
\equiv p^k\cdot \sum_{\bar{x}=0}^{p^k-1} \alpha_k(\bar{x}) + \sum_{\bar{x}=0}^{p^k-1} f_{k-1} (\bar{x}) \bmod p^{k+1}.
\end{align*}

As $f_{k-1}$ is bijective on $\Z/p^{k-1},$ then $\sum_{\bar{x}=0}^{p^k-1} f_{k-1} (\bar{x})=\frac{p^k(p^k-1)}{2},$ i.e.
$$\sum_{x=0}^{p^k-1} f(x)\equiv p^k\cdot \sum_{\bar{x}=0}^{p^k-1} \alpha_k(\bar{x})+ \frac{p^k(p^k-1)}{2} (\bmod p^{k+1}),$$ or
\begin{align*}
&\frac{1}{p^k} \cdot \sum_{x=0}^{p^k-1} f(x)\equiv \sum_{\bar{x}=0}^{p^k-1} \alpha_k(\bar{x})+ \frac{p^k-1}{2}\equiv \\
&\equiv \sum_{\bar{x}=0}^{p^k-1} \alpha_k(\bar{x})-2^{-1} (\bmod p).
\end{align*}
Note that $2^{-1}=2^{p-2} \bmod p.$ Then by the second condition of this Theorem 
$$\sum_{\bar{x}=0}^{p^k-1} \alpha_k(\bar{x})=\frac{1}{p^k} \cdot \sum_{x=0}^{p^k-1} f(x)+2^{p-2} \neq 0(\bmod p).$$
And, therefore, the permutation $F_{k,0}$ is transitive on $\Z/p\Z.$ This means that the function $f_k$ is transitive on $\Z/p^k.$
By Theorem 4.23 [p.99, \cite{ANKH}] from transitivity of $f_k$ on $\Z/p^k,$ $k=0,1,2,\ldots,$ follows that the function $f$ is ergodic.

Vice versa, let the function $f$ be ergodic function. Then $f \bmod p\equiv \varphi_0$ and functions $F_{k,0},$ $k=1,2,\ldots$ are transitive on $\Z/p\Z$ by Theorem \ref{generg2}.
As $F_{k,0}=x_k+ \sum_{i=0}^{p^k-1} \alpha_k(f_{k-1}^{(i)}(0))=x_k+ \sum_{i=0}^{p^k-1} \alpha_k (i),$ then transitivity of $F_{k,0}$ means that 
$\sum_{i=0}^{p^k-1} \alpha_k (i)\neq 0 \bmod p.$
Then $$\frac{1}{p^k} \cdot \sum_{x=0}^{p^k-1} f(x)\equiv \sum_{\bar{x}=0}^{p^k-1} \alpha_k(\bar{x})+ \frac{p^k-1}{2}\equiv \sum_{\bar{x}=0}^{p^k-1} \alpha_k(\bar{x})-2^{-1} (\bmod p),$$
i.e. $2^{p-2}+\frac{1}{p^k} \cdot \sum_{i=0}^{p^k-1} f(i) \neq 0 \bmod p,$ $k=2,3,\ldots.$
\end{proof}

\begin{comment}
In the Comment \ref{cicl} was stated that if for each $k=1,2,\ldots$ all permutations $\varphi_{k,\bar{x}_{k-1}},$ $\bar{x}_{k-1} \in \left\{0,\ldots,p^{k-1}-1\right\}$
belong to  a cyclic subgroup of the symmetric group $S_p,$
then the conditions of ergodicity of $f$ are simplified.
In Theorem \ref{poln} we considered the case such that the cyclic group is generated by permutation on $\Z/p\Z,$ which is given by a linear polynomial
$x+c,$ $c\neq 0.$In this case,  the ergodicity of $f$ is determined by the value of $\frac{1}{p^k} \cdot \sum_{i=0}^{p^k-1} f(i) \bmod p,$ $k=2,3,\ldots.$

Note that such sums appear in the definition of Volkenborn's integral, see, for example, Definition 55.1 \cite{Schikhof}.
However, real coupling between theory of Volkenborn integration and ergodicity of $p$-adic dynamical systems has not yet been clarified.
\end{comment}

\begin{remark}
The condition $p\not=2$ is important. Theorem \ref{poln} cannot be directly generalized to the case $p=2.$ At the same time
the authors expect that even in the case $p=2$ some kind of condition in terms of sums of values of the function $f$ can be 
found. For a moment, this is an open problem.
 
\end{remark}

The functions satisfying conditions of Theorem \ref{poln}  can be presented without using their coordinate representation.

Let us present the explicit form of $p$-adic functions, whose coordinate functions have the following form:
$\varphi_k(x_0,x_1,\ldots,x_k)=x_k+\alpha_k(x_0,x_1,\ldots,x_{k-1})=x_k+\alpha_k(\bar{x}),$ $k=1,2,\ldots .$
In other words, all permutations $\varphi_{k,\bar{x}_{k-1}} \in S_p,$ $\bar{x}_{k-1}\in \left\{0,\ldots,p^{k-1}-1\right\},$ $k=2,3,\ldots $
 are given by linear polynomials $x+c$ over $\Z/p\Z.$

\begin{lemma}
\label{bezcoord}
Let $p$-adic ($p\neq 2$) compatible function $f\:\Z_p\>\Z_p$ be presented in the coordinate form (\ref{coordf}). 
Then $\varphi_k(x_0,x_1,\ldots, x_k)=x_k+\alpha_k(x_0,\ldots, x_{k-1})=x_k+\alpha_k(\bar{x}),$
$k=1,2,\ldots,$  if and only if $$f(x)=f(x_0+px_1+\ldots +p^kx_k+ \ldots)=\varphi_0(x_0)+(x-x_0)+p\cdot g(x),$$
where $g\:\Z_p\> \Z_p$ is an arbitrary compatible function.
\end{lemma}

\begin{proof}
Let $f(x)=\sum_{m=0}^{\infty} B_m \chi(m,x)=\sum_{m=0}^{\infty} p^{\left\lfloor \log_p m \right\rfloor} b_m \chi(m,x)$ be the van der Put representation of the function $f.$
Let us find values of the coefficients $b_m$ for $m\geq p.$
Denote $m=x_0+px_1+\ldots +p^kx_k=\bar{x}+p^kx_k,$ $\bar{x}\in \left\{0,\ldots,p^k-1\right\},$ $x_k\in \left\{1,\ldots,p-1\right\}.$
Then 
\begin{align*}
&B_m=B_{\bar{x}+p^kx_k}=f(\bar{x}+p^kx_k)-f(\bar{x})= \\
&= p^k(\varphi_k(\bar{x}+p^kx_k)-\varphi_k(\bar{x}))+p^{k+1}(\ldots),
\end{align*}
i.e. 
\begin{align*}
&b_m=b_{\bar{x}+p^kx_k}\equiv \varphi_k(\bar{x}+p^kx_k)-\varphi_k(\bar{x})\equiv \\
&\equiv x_k+ \alpha_k(\bar{x})- \alpha_k(\bar{x})\equiv x_k \bmod p.
\end{align*}
In other words, $b_{\bar{x}+p^kx_k}= x_k+p\check{b}_{\bar{x}+p^kx_k}=x_k+p\check{b}_m $
for suitable $p$-adic integer $\check{b}_m=\check{b}_{\bar{x}+p^kx_k}.$

We set $f(i)=B_i=b_i\equiv b_i \bmod p+p\check{b}_i,$ $\check{b}_i\in \Z_p,$ $i=\left\{0,\ldots,p-1\right\}.$
Let $\varphi_0\: \Z/p\Z\> \Z/p\Z,$ where $\varphi_0\equiv f(x) \bmod p,$ i.e. $\varphi_0$ is a coordinate function from the representation of $f$ in the coordinate form.
Note that $\varphi_0(x)=b_0 (\bmod p) \chi(0,x)+\ldots +b_{p-1} (\bmod p) \chi(p-1,x)$ and
\begin{align*}
&\sum_{m=p}^\infty b_m (\bmod p) \cdot p^{\left\lfloor \log_p m \right\rfloor} \cdot \chi(m,x)= \\
&=\sum_{m=p}^\infty q(m) \chi(m,x)=x-x_0,
\end{align*}
where $$q (m)=q(m_0+pm_1+\ldots +p^{\left\lfloor \log_p m \right\rfloor}m_{\left\lfloor \log_p m \right\rfloor}=p^{\left\lfloor \log_p m \right\rfloor}m_{\left\lfloor \log_p m \right\rfloor},$$ $m_{\left\lfloor \log_p m \right\rfloor}\neq 0$ and $x=x_0+px_1+\ldots +p^kx_k+\ldots.$

And let us consider a compatible function $g(x)=\sum_{m=0}^{\infty} \check{b}_m p^{\left\lfloor \log_p m \right\rfloor}  \chi(m,x).$
Thus,
\begin{align*}
&f(x)=\sum_{m=0}^{\infty} b_m p^{\left\lfloor \log_p m \right\rfloor}  \chi(m,x)=\sum_{m=0}^\infty b_m (\bmod p) \cdot p^{\left\lfloor \log_p m \right\rfloor} \cdot \chi(m,x)+ \\
&+p\sum_{m=0}^{\infty} \check{b}_m p^{\left\lfloor \log_p m \right\rfloor}  \chi(m,x)= \sum_{m=0}^{p-1} b_m (\bmod p) p^{\left\lfloor \log_p m \right\rfloor}  \chi(m,x) +\\
&+\sum_{m=p}^\infty b_m (\bmod p) \cdot p^{\left\lfloor \log_p m \right\rfloor} \cdot \chi(m,x) +pg(x)= \\
&=\varphi_0+(x-x_0)+pg(x).
\end{align*}

Vice versa, let $f(x)=f(x_0+px_1+\ldots +p^kx_k+\ldots)=\varphi_0+(x-x_0)+pg(x),$  
$f(x_0+px_1+\ldots +p^kx_k+\ldots)=\varphi_0(x_0)+p\cdot \varphi_1(x_0,x_1)+\ldots+p^k\cdot \varphi_k(x_0,x_1,\ldots,x_k)+\ldots,$
and $f(x)=\sum_{m=0}^{\infty} b_m p^{\left\lfloor \log_p m \right\rfloor}  \chi(m,x).$

Set $m=x_0+px_1+\ldots +p^kx_k=\bar{x}+p^kx_k,$ we obtain that for $m\geq p$
\begin{align*}
&b_m=b_{\bar{x}+p^kx_k}\equiv \frac{1}{p^k} \left(f(\bar{x}+p^kx_k)-f(\bar{x})\right)\equiv \\
&\equiv \frac{1}{p^k}(\varphi_0(x_0)+(\bar{x}+p^kx_k-x_0)+\\
&+p\cdot g(\bar{x}+p^kx_k)-(\varphi_0(x_0)+(\bar{x}-x_0)+p\cdot g(\bar{x})))\equiv \\
&\equiv \frac{1}{p^k} \left(p^kx_k+p\cdot(g(\bar{x}+p^kx_k)-g(\bar{x}))\right)\equiv x_k+ \frac{1}{p^{k-1}} (g(\bar{x}+p^kx_k)-g(\bar{x})) \bmod p.
\end{align*}

As the function $g$ is compatible, then $g(\bar{x}+p^kx_k)\equiv g(\bar{x}) \bmod p^k$ and
$\frac{1}{p^{k-1}} (g(\bar{x}+p^kx_k)-g(\bar{x}))\equiv 0 \bmod p.$
Then $b_m\equiv b_{\bar{x}+p^kx_k}\equiv x_k \bmod p.$
On the other hand,
$b_m\equiv b_{\bar{x}+p^kx_k}\equiv \varphi_k(\bar{x}+p^kx_k)-\varphi_k(\bar{x}) \bmod p,$
and, therefore, for any fixed $\bar{x}=x_0+\ldots +p^{k-1}x_{k-1}$ we have
$$\varphi_k(\bar{x}+p^kx_k)-\varphi_k(\bar{x}) \equiv \varphi_k(x_0,\ldots ,x_{k-1},x_k)-\varphi_k(x_0,\ldots ,x_{k-1},0) \equiv x_k \bmod p.$$
Then $\varphi_k(x_0,\ldots ,x_{k-1},x_k)\equiv x_k+\varphi_k(x_0,\ldots ,x_{k-1},0) \bmod p.$
\end{proof}

\begin{theorem}
\label{prop3}
Let $p$-adic ($p\neq 2$) compatible function $f\:\Z_p\>\Z_p$ have the form
$f(x)=f(x_0+px_1+\ldots +p^kx_k+\ldots)=\varphi_0(x_0)+(x-x_0)+pg(x),$
where $g\:\Z_p\> \Z_p$  is a compatible function and $\varphi_0\: \Z/p\Z\longmapsto \Z/p\Z.$

The function $f$ is ergodic if and only if
\begin{enumerate}
\item $\varphi_0$ is transitive  (monocycle) permutation on the set  of residues modulo $p;$
\item $\frac{1}{p^{k-1}} \sum_{i=0}^{p^k-1} g(i) \neq 2^{p-2} \bmod p,$ $k=2,3,\ldots .$
\end{enumerate}
\end{theorem}

\begin{proof}
Consider  coordinate representation of the function $f.$ From the Lemma \ref{bezcoord} it follows that the functions $\varphi_k$ have the form of linear polynomials, i.e. 
$\varphi_k(x_0,x_1,\ldots, x_k)=x_k+\alpha_k(x_0,\ldots, x_{k-1})=x_k+\alpha_k(\bar{x}),$ $k=1,2,\ldots .$
Thus the function $f$ satisfies the conditions of Theorem \ref{poln}. For the proof  it is sufficient to obtain an expression for
$2^{p-2}+\frac{1}{p^k} \cdot \sum_{i=0}^{p^k-1} f(i)  \bmod p,$ $k=2,3,\ldots$ and use Theorem \ref{poln}.
Setting $x=x_0+px_1+\ldots +p^{k-1}x_{k-1},$ we see that for $k=2,3,\ldots$
\begin{align*}
&2^{p-2}+\frac{1}{p^k} \cdot \sum_{x=0}^{p^k-1} f(x) \equiv \\
&\equiv 2^{p-2}+\frac{1}{p^k} \left(p^{k-1}\sum_{x_0=0}^{p-1} \varphi_0(x_0) + \sum_{x=0}^{p^k-1} (x-x_0) + p\sum_{x=0}^{p^k-1} g(x)\right) \equiv \\
&\equiv 2^{p-2}+ \\
&+\left(\frac{p-1}{2}+\frac{p^{k-1}-1}{2}\right) + \frac{1}{p^{k-1}} \sum_{x=0}^{p^k-1} g(x)\equiv \\
&\equiv -2^{p-2}+\frac{1}{p^{k-1}} \sum_{x=0}^{p^k-1} g(x) \bmod p.
\end{align*}
\end{proof}

\begin{corollary}[see Lemma 4.41, p.112, \cite{ANKH}]
\label{leman}
Let $f\:\Z_p\>\Z_p,$ $f(x)=c+x+p(h(x+1)-h(x)),$ where $h\:\Z_p\>\Z_p$ is compatible function and $p\neq 2.$ If $c\neq 0 \bmod p,$ then $f$ is an ergodic function.
\end{corollary}

\begin{proof}
As $f(x)\bmod p\equiv x_0+c \; (\bmod \; p),$ where $x=x_0+px_1+\ldots ,$ then $f(x)\bmod \; p$ is transitive permutation once $c\neq 0 \bmod \; p.$
From compatibility of $h$ follows that $h(p^k)-h(0)\equiv 0 \bmod p^k.$ Then for $k=2,3,\ldots$
$$\frac{1}{p^{k-1}} \sum_{i=0}^{p^k-1} (h(i+1)-h(i))\equiv \frac{h(p^k)-h(0)}{p^k-1}\equiv 0\neq 2^{p-2} (\bmod p).$$
Thus the function $f$ is ergodic by  Theorem \ref{prop3}.
\end{proof}

\begin{corollary}[Corollary 4.42, p.113, \cite{ANKH}]
\label{leman2}
Let $f\:\Z_p\>\Z_p,$ $f(x)=c+r\cdot x+p(h(x+1)-h(x)),$ where $h\:\Z_p\>\Z_p$ is a compatible function and $p\neq 2.$ 
If $c\neq 0 \bmod p$ and $r\equiv 1 \bmod p,$ then $f$ is an ergodic function.
\end{corollary}

\begin{proof}
As $f(x)\bmod\;  p\equiv x_0+c (\bmod \; p),$ where $x=x_0+px_1+\ldots ,$ then $f(x)\bmod p$ is transitive permutation once $c\neq 0 \bmod\;  p$ and $r\equiv 1 \bmod \; p.$
From compatibility of $h$ follows that $h(p^k)-h(0)\equiv 0 \bmod p^k.$ Then for $k=2,3,\ldots$ and $r=1+p\cdot \check{r}$ we have
\begin{align*}
&2^{p-2}+\frac{1}{p^k}\cdot \sum_{x=0}^{p^k-1} f(x) \equiv 2^{p-2}+\frac{1}{p^k}\left(c\cdot p^k+r\cdot\frac{p^k(p^k-1)}{2} + p(h(p^k)-h(0))\right)\equiv \\
&\equiv 2^{p-2}+c+ r\cdot\frac{p^k-1}{2} \equiv c-p\cdot \check{r}\equiv c \bmod p.
\end{align*}

By assumption, $c\neq 0 \bmod p,$ therefore, $2^{p-2}+\frac{1}{p^k}\cdot \sum_{x=0}^{p^k-1} f(x) \neq 0 \bmod p,$ $k=2,3,\ldots .$
Then the function $f$ is ergodic by Theorem \ref{poln}.
\end{proof}

Thus Theorems \ref{poln} and \ref{prop3} describe  ergodic functions from the following class. We consider such functions, in which the coordinate functions have the form
$\varphi_k(x_0,x_1,\ldots, x_k)=x_k+\alpha_k(x_0,\ldots, x_{k-1})=x_k+\alpha_k(\bar{x}),$ $k=1,2,\ldots .$
In other words, here all permutations $\varphi_{k,\bar{x}_{k-1}} \in S_p,$ $\bar{x}_{k-1}\in \left\{0,\ldots,p^{k-1}-1\right\},$ $k=2,3,\ldots $
 are given by linear polynomials $x+c$ over $\Z/p\Z.$

By Lemma \ref{bezcoord} such functions can be represented in the equivalent form
$f(x)=f(x_0+px_1+\ldots +p^kx_k+ \ldots)=\varphi_0(x_0)+(x-x_0)+p\cdot g(x),$
where $g\:\Z_p\> \Z_p$ - compatible function.

In Corollaries \ref{leman} and \ref{leman2} we considered the case where the function $g\:\Z_p\>\Z_p$  has the representation $g(x)=h(x+1)-h(x),$
where $h\:\Z_p\>\Z_p$ is compatible function.
For such functions (see, for example, \cite{ANKH}) conditions of  ergodicity were obtained. In  Corollaries \ref{leman} and \ref{leman2} we present new proof of these results.

Note that  the functions from Corollaries \ref{leman} and \ref{leman2} (or [ANKH])  contained in the class of functions considered in Theorems \ref{poln} and \ref{prop3}.
In other words, Theorems \ref{poln} and \ref{prop3} generalize the results of  \cite{ANKH}.

Theorem 9.20 (part 2) [p.286, \cite{ANKH}] describes  compatible ergodic $2$-adic functions of the form:
$f(x_0+2x_1+\ldots+2^kx_k+\ldots)=c+a_0\cdot x_0+a_1\cdot 2x_1+\ldots+a_k\cdot 2^kx_k+\ldots ,$ where $a_k\in \Z_2,$ $k=0,1,\ldots,$ $c\in \Z_2.$
Note that the coefficients before $x_k$ provide compatibility of the function $f.$

The next statement describes the  $p$-adic functions of such form for all odd $p.$

\begin{theorem}
\label{ergan}
Let  $f\:\Z_p\>\Z_p$ ($p\neq 2$) be a compatible function of the form
$$
f(x_0+px_1+\ldots+p^kx_k+\ldots)=c+a_0\cdot x_0+a_1\cdot px_1+\ldots+a_k\cdot p^kx_k+\ldots,
$$ where $a_k\in \Z_p,$ $k=0,1,\ldots,$ $c\in \Z_p.$
The function $f$ is ergodic if and only if $c\neq 0 \bmod p$ and $a_k\equiv 1 \bmod p,$ $k=0,1,\ldots .$
\end{theorem}

\begin{proof}
Let a function $f\:\Z_p\>\Z_p$ be represented in the coordinate form. Then for suitable $p$-valued function $l_k(x_0,x_1,\ldots, x_{k-1})$
\begin{align*}
 \varphi_k(x_0,x_1,\ldots, x_k)&\equiv \frac{1}{p^k} (f(x_0+px_1+\ldots +p^kx_k)- \\
 &-f(x_0+px_1+\ldots +p^{k-1}x_{k-1}) (\bmod p^k)) \bmod p \equiv \\
& \equiv a_kx_k+l_k(x_0,x_1,\ldots, x_{k-1}).
\end{align*}
Let $\bar{x}_{k-1}= (x_0,\ldots, x_{k-1}).$ For any fixed $\bar{x}_{k-1}$ we consider functions \\
$\varphi_{k,\bar{x}_{k-1}}\: \Z/p\Z \> \Z/p\Z$ defined so that
$\varphi_{k,\bar{x}_{k-1}} (x)=\varphi (x,\bar{x}_{k-1}).$
Suppose that the function $f$ is ergodic. By Theorem \ref{generg2} functions
$F_{k,0}=\varphi_{k,f_{k-1}^{(p^k-1)}(0)} \circ  \varphi_{k,f_{k-1}^{(p^k-2)}(0)}\circ\ldots\circ \varphi_{k,0},$ $k=1,2,\ldots $ and $f_0=\varphi_0$
are transitive permutations on $\Z/p\Z,$ where $f_k\equiv f \bmod p^{k+1}.$
As for suitable $\beta_k\in \Z/p\Z,$ $k=1,2,\ldots $
\begin{align*}
F_{k,0}&=\left(a_k x_k+l_k(f_{k-1}^{(p^{k-1}-1)}(0))\right) \circ\ldots\circ \left(a_k x_k+l_k (0)\right)= \\
&=a_k^{p^k}\cdot x_k+\beta_k,
\end{align*}
then $a_k^{p^k}\equiv a_k\equiv 1 \bmod p.$
From transitivity of $\varphi_0=f_0\equiv f \bmod p \equiv c+a_0x_0 \bmod p$ follows that $c\neq 0 \bmod p$ and $a_0\equiv 1 \bmod p.$

Vice versa, by assumption $a_k\equiv 1 \bmod p,$ $k=1,2,\ldots .$ Then \\
$\varphi_k(x_0,x_1,\ldots, x_k)=x_k+l_k(x_0,x_1,\ldots, x_{k-1})$ for suitable $p$-valued function \\
$l_k(x_0,x_1,\ldots, x_{k-1}).$
This, in  particular, means that the function $f$ satisfy the conditions of Theorem \ref{poln}. Let us check the function $f$ on ergodicity.
As $c\neq 0 \bmod p$ and $a_0\equiv 1 \bmod p,$ then $\varphi_0=f_0\equiv f \bmod p \equiv c+x_0 \bmod p$ is transitive on $\Z/p\Z.$
Note that for $2^{-1}$ - an inverse to the residue $2$ in $\Z/p\Z,$
\begin{align*}
&\frac{1}{p^k} \sum_{i=0}^{p^k-1} f(i)=\frac{1}{p^k} \sum_{x_0,\ldots, x_{k-1}} \left(c+a_0x_0+a_1px_1+\ldots +a_{k-1}\cdot p^{k-1} x_{k-1}\right) = \\
&=\frac{1}{p^k}\left(p^kc+a_0\cdot p^{k-1}\frac{p(p-1)}{2} + a_1\cdot p^k\frac{p(p-1)}{2} +\ldots + a_{k-1}\cdot p^{2k-2}\frac{p(p-1)}{2}\right) = \\
&=c+a_0 \frac{p-1}{2} =c-a_0\cdot 2^{-1}\equiv c- 2^{-1} \bmod p.
\end{align*}
Therefore, $2^{-1}+\frac{1}{p^k} \sum_{i=0}^{p^k-1} f(i)\equiv c \neq 0 \bmod p.$ Then by Theorem \ref{poln}, the function $f$ is ergodic.
\end{proof}

\begin{theorem}
\label{ergfix}
Let $p$-adic ($p\neq 2$) compatible and measure-preserving function $f\:\Z_p\>\Z_p$ be presented in the coordinate form 
(\ref{coordf})
,
where $\varphi_k(x_0,x_1,\ldots, x_k)$ are such $p$-valued functions that 
$$\varphi_k(x_0,\ldots, x_k)=x_kA_k(x_0,\ldots,x_{S-1}) + \alpha_k(x_0,\ldots,x_{k-1})=x_kA_k(\bar{x}_{S-1})+\alpha_k(\bar{x}_{k-1}),$$
$k=S,S+1,\ldots$ for some fixed integer $S.$

The function $f$ is ergodic if and only if holds simultaneously
\begin{enumerate}
\item $f_{S-1}\equiv f \bmod p^S$ be transitive  (monocycle) permutation on the set  of residues modulo $p^S;$
\item $\prod_{i=0}^{p^S-1} A_k(i)\equiv 1 \bmod p,$ $k=S,S+1,\ldots;$
\item for $k=S+1,S+2,\ldots$ 
\begin{align*}
&2^{-1} \sum_{i=0}^{p^S-1} \tau_k^{(i)}- \sum_{i=0}^{p^S-1} \frac{\tau_k^{(i)}}{p^k} [p^{k-S}(f^{(i+1)}(0) \bmod p^S)- \\
&-\sum_{\beta=0}^{p^{k-S}-1} f(f^{(i)}(0) \bmod p^S+p^S\beta) ]\neq 0 \bmod p,
\end{align*}
and for $k=S$
$$\sum_{i=0}^{p^S-1} \frac{\tau_S^{(i)}}{p^S} \left[f(f^{(i)}(0) \bmod p^S) - f^{(i+1)}(0) \bmod p^S\right] \neq 0 \bmod p,$$
where $\tau_k^{(i)}=\prod_{j=i+1}^{p^S-1} A_k(f^{(j)}(0)),$ $i=0,1,\ldots,p^k-2$ and $\tau_k^{(i)}=1$ for $i=p^k-1.$
\end{enumerate}
\end{theorem}

\begin{proof} As always, we denote as
Via $\varphi_{k,\bar{x}_{k-1}}\: \Z/p\Z \> \Z/p\Z,$ $k=1,2,\ldots,$ the functions obtained 
from  the coordinate functions $\varphi_k(x_0,x_1,\ldots, x_k)$ by fixing the first $k-1$variables 
$x_0,x_1,\ldots, x_{k-1}.$ We will assume that the vector $(x_0,x_1,\ldots, x_{k-1})$ sets a residue  modulo $p^k,$ i.e. $x_0+px_1+\ldots +p^{k-1}x_{k-1}\in \left\{0,1,\ldots , p^k-1\right\}.$ 
 Since $f$ preserves the measure, then by Theorem \ref{measurecoord} functions $\varphi_{k,\bar{x}_{k-1}}$ and $\varphi_0$ are permutations on $\Z/p\Z.$
Note that for $k=S,S+1,\ldots$
\begin{align*}
&F_{k,0}=\varphi_{k,f_{k-1}^{(p^k-1)}(0)} \circ  \varphi_{k,f_{k-1}^{(p^k-2)}(0)}\circ\ldots\circ \varphi_{k,0}=\\
&\left(A_k(f_{k-1}^{(p^k-1)}(0))\cdot x_k + \alpha_k(f_{k-1}^{(p^k-1)}(0))\right) \circ\ldots\circ \left(A_k(0)\cdot x_k+\alpha_k(0)\right)= \\
&x_k\left(\prod_{i=0}^{p^k-1} A_k(f_{k-1}^{(i)}(0))\right)+ \\
&+\left(\alpha_k(f_{k-1}^{(p^{k-1}-1)}(0)) + \sum_{i=0}^{p^{k-1}-2} \alpha_k (f_{k-1}^{(i)}(0))\cdot \left(\prod_{j=i+1}^{p^k-1} A_k(f_{k-1}^{(j)}(0))\right)\right)= \\
&=x_k\cdot Q_k+L_k.
\end{align*}

By induction on the $k=S,S+1,\ldots$ we show that $F_{k,0}=\varphi_{k,f_{k-1}^{(p^k-1)}(0)} \circ\ldots\circ \varphi_{k,0}$ are transitive permutations on $\Z/p\Z.$

From compatibility and transitivity of the function $f$ modulo $p^S$ follows that $f_{S-1}^{(p^S)}(0)\equiv 0 \bmod p^S$ and $f_{S+m-1}^{(p^{m+S})}(0)\equiv 0 \bmod p^S,$
$m=1,2,\ldots$ for $k=S, S+1,\ldots .$
As by assumption $p$-valued functions $A_k$  depend only on variables $x_0,x_1,\ldots, x_{S-1},$ then any sequence 
$\left\{A_k(0),A_k(f_{k-1}(0)),\ldots,A_k(f_{k-1}^{(p^k-1)}(0))\right\},$ $k=S,S+1,\ldots$ has period of the length $p^S.$
Then from the second condition of Theorem follows that
$$Q_k=\prod_{i=0}^{p^k-1} A_k(f_{k-1}^{(i)}(0))=\prod_{i=0}^{p^k-1} A_k(i)=\left(\prod_{i=0}^{p^S-1} A_k(i)\right)^{p^{k-S}}=1.$$
And 
\begin{align*}
L_k&= \alpha_k(f_{k-1}^{(p^k-1)}(0)) + \sum_{i=0}^{p^k-2} \alpha_k (f_{k-1}^{(i)}(0))\cdot \left(\prod_{j=i+1}^{p^k-1} A_k(f_{k-1}^{(j)}(0))\right)= \\
&=\sum_{\alpha=0}^{p^{k-S}-1} \alpha_k(f_{k-1}^{(p^S-1+P^S\alpha)}(0))+ \\
&+ \sum_{i=0}^{p^S-2} \left(\prod_{j=i+1}^{p^S-1} A_k(f_{k-1}^{(j)}(0))\right) \left(\sum_{\alpha=0}^{p^{k-S}-1} \alpha_k(f_{k-1}^{(i+P^S\alpha)}(0))\right).
\end{align*}

Let $\tau_k^{(i)}=\prod_{j=i+1}^{p^S-1} A_k(f_{k-1}^{(j)}(0)),$ $i=0,1,\ldots,p^{k-1}-2,$ $k=S,S+1,\ldots,$ then
$$L_k=\sum_{\alpha=0}^{p^{k-S}-1} \alpha_k(f_{k-1}^{(p^S-1+P^S\alpha)}(0)) + \sum_{i=0}^{p^S-2} \tau_k^{(i)} \left(  \sum_{\alpha=0}^{p^{k-S}-1} \alpha_k(f_{k-1}^{(i+P^S\alpha)}(0))\right).$$
Thus, $F_{k,0}=x_k+L_k,$ $k=S,S+1,\ldots.$

Suppose we have already shown that $F_{k-1,0}$ is transitive on $\Z/p\Z,$ and $f_{k-1}\equiv f \bmod p^k$ is transitive on $\Z/p^k\Z.$
Let us check that $F_{k,0}=x_k+L_k$ is transitive on $\Z/p\Z.$
It is enough to show that $L_k\neq 0.$ As $f$ is a compatible function and by the induction $f_{k-1}\equiv f \bmod p^k$ is transitive on $\Z/p^k\Z,$
then for $i\in \left\{0,1,\ldots,p^S-1\right\}$
\begin{eqnarray}
\label{eq1}
\left\{f_{k-1}^{(i+P^S\alpha)}(0) \bmod p^k|\alpha=0,1,\ldots,p^{k-S}-1\right\}= \nonumber\\
=\left\{f_{k-1}^{(i)}(0) \bmod p^S +p^S\cdot \beta|\beta=0,1,\ldots,p^{k-S}-1\right\},
\end{eqnarray}
and $f_{k-1}^{(i)}(0)\equiv f_{S-1}^{(i)}(0) \bmod p^S.$ Then
\begin{align*}
&\sum_{\alpha=0}^{p^{k-S}-1} \alpha_k(f_{k-1}^{(i+P^S\alpha)}(0))= \sum_{\beta=0}^{p^{k-S}-1} \alpha_k\left(f_{S-1}^{(i)}(0) \bmod p^S + p^S\beta\right)=\\
&=\sum_{\beta=0}^{p^{k-S}-1} \alpha_k\left(f^{(i)}(0) \bmod p^S + p^S\beta\right)
\end{align*}
and
\begin{align*}
L_k&= \sum_{\alpha=0}^{p^{k-S}-1}  \alpha_k\left(f^{(p^S-1)}(0) \bmod p^S + p^S\alpha\right) +\\
&+  \sum_{i=0}^{p^S-2} \tau_k^{(i)} \left( \sum_{\alpha=0}^{p^{k-S}-1}  \alpha_k\left(f^{(i}(0) \bmod p^S + p^S\alpha\right)\right).
\end{align*}

As by assumption $A_k$ depend only on variables $x_0,x_1,\ldots,x_{S-1},$ then
$$A_k(f_{k-1}^{(i)}(0))=A_k(f_{k-1}^{(j)}(0) \bmod p^S)=A_k(f_{S-1}^{(j)}(0)).$$
And, therefore, for  $i\in \left\{0,1,\ldots,p^{k-1}-2\right\}$
$$\tau_k^{(i)}=\prod_{j=i+1}^{p^S-1} A_k(f_{S-1}^{(j)}(0))=\prod_{j=i+1}^{p^S-1} A_k(f^{(j)}(0)).$$

For the case $k>S$ we introduce auxiliary notations
\begin{align*}
&\bar{\Delta}=(\Delta_0,\Delta_1,\ldots,\Delta_{S-1}),  \bar{\beta}=(\beta_0,\beta_1,\ldots,\beta_{k-S-1}), \\
&\Delta=\Delta_0+p\Delta_1+\ldots+p^{S-1}\Delta_{S-1} \in \left\{0,1,\ldots,p^S-1\right\}, \\
&\beta=\beta_0+p\beta_1+\ldots+p^{k-S-1}\beta_{k-S-1} \in \left\{0,1,\ldots,p^{k-S-1}-1\right\}.
\end{align*}

Using coordinate representation of the function $f$ and relation \ref{eq1} we obtain  
\begin{align*}
&\sum_{\beta=0}^{p^{k-S}-1} f(\Delta+p^S\beta)= \sum_{\beta=0}^{p^{k-S}-1} \left(\sum_{i=0}^{k-1} pî\varphi_i+p^k\varphi_k(\bar{\Delta},\bar{\beta},0)\right)= \\
&=\sum_{\beta=0}^{p^{k-S}-1} f(\Delta+p^S\beta) \bmod p^k + p^k\sum_{\beta=0}^{p^{k-S}-1} \alpha_k(\Delta+p^S\beta)\equiv \\
&\equiv p^{k-S}f(\Delta)\bmod p^S + p^S\frac{p^{k-S}(p^{k-S}-1)}{2} + p^k\sum_{\beta=0}^{p^{k-S}-1} \alpha_k(\Delta+p^S\beta) \bmod p^{k+1}.
\end{align*}

Then, for $2^{-1}=\frac{1-p}{2},$ $p\neq 2,$ we obtain
\begin{eqnarray}
\label{eq2}
&\sum_{\beta=0}^{p^{k-S}-1} \alpha_k(\Delta+p^S\beta)\equiv \nonumber\\
&\equiv 2^{-1}-\frac{1}{p^k} \left(p^{k-S}(f(\Delta)\bmod p^S) - \sum_{\beta=0}^{p^{k-S}-1} f(\Delta+p^S\beta)\right) \bmod p.
\end{eqnarray}

In the case $k=S,$ 
$$f(\Delta)=\sum_{i=0}^{S-1} p^i \varphi_i+p^S\varphi_S(\bar{\Delta},0)\equiv f(\Delta) \bmod p^S + p^S\alpha_S(\Delta)\bmod p^{S+1},$$
or, respectively, 
\begin{equation}
\label{eq3}
\alpha_S(\Delta)= \left(f(\Delta)-f(\Delta) \bmod p^S \right) \bmod p.
\end{equation}

This relation we will use in the proof of  the base of induction $k=S.$

Thus, substituting \ref{eq2} in the relation for $L_k,$ and assuming $\tau_k^{(p^S-1)}=1,$ we obtain
\begin{align*}
&L_k=2^{-1}\cdot \sum_{i=0}^{p^S-1} \tau_k^{(i)} - \sum_{i=0}^{p^S-1} \frac{\tau_k^{(i)}}{p^k} \cdot \\
&\cdot \left[p^{k-S} (f^{i+1}(0) \bmod p^S) - \sum_{\beta=0}^{p^{k-S}-1} f(f^{i}(0) \bmod p^S +p^S\beta)\right] \bmod p.
\end{align*}

Then by the third condition of this Theorem $L_k\neq 0.$ And the function $F_{k,0}=x_k+L_k$ is transitive on $\Z/p\Z,$ and moreover,
$f_k\equiv f \bmod p^{k+1}$ is transitive on $\Z/p^{k+1}\Z.$

Let us show that $F_{S,0}$ is transitive on $\Z/p\Z,$ and $f_S\equiv f \bmod p^{S+1}$ is transitive on $\Z/p^{S+1}\Z.$
As stated above, $F_{S,0}=x_S+L_S,$ where for $i=0,1,\ldots,p^{S-1}-2$
$$L_S= \alpha_S(f_{S-1}^{(p^S-1}(0)) + \sum_{i=0}^{p^{S-1}-2} \tau_S^{(i)} (\alpha_S(f_{S-1}^{(i)}(0))),$$
$$\tau_S^{(i)}=\prod_{j=i+1}^{p^S-1} A_S(f^{(j)}(0)).$$ 

Set $\tau_S^{(p^S-1)}=1,$ and from relation \ref{eq3}, we obtain
$$L_S= \sum_{i=0}^{p^S-1} \frac{\tau_S^{(i)}}{p^S} \cdot \left[f\left( f^{i}(0) \bmod p^S\right) - f^{i+1}(0) \bmod p^S\right] \bmod p.$$

From the third condition of this Theorem follows that $L_S\neq 0,$ i.e. $F_{S,0}=x_S+L_S$ is transitive on $\Z/p\Z.$
As by initial conditions $f_{S-1}\equiv f \bmod p^S$ is transitive on $\Z/p^S\Z$ and $F_{S,0}$ is transitive on $\Z/p\Z,$ 
then $f_S\equiv f \bmod p^{S+1}$ is transitive on $\Z/p^{S+1}\Z.$

Thus, permutations $F_{k,0}=\varphi_{k,f_{k-1}^{(p^k-1)}(0)} \circ  \varphi_{k,f_{k-1}^{(p^k-2)}(0)}\circ\ldots\circ \varphi_{k,0}$ are transitive on $\Z/p\Z$ for $k=S,S+1,\ldots .$

Let us show that $\varphi_0$ and $F_{k,0},$ $k\in \left\{1,2,\ldots,S-1\right\}$  are transitive on $\Z/p\Z.$
As $f$ is compatible function and $f_{S-1}\equiv f \bmod p^S$ is transitive on $\Z/p^S\Z,$ then $f_k\equiv f \bmod p^{k+1},$ $k\in \left\{0,\ldots,S-2\right\}$  is transitive on $\Z/p^{k+1}\Z.$
Then $f_k^{(p^k)} (p^kx_k)=p^kF_{k,0}(x_k),$ $x_k\in \left\{0,\ldots,p-1\right\},$ $k=1,\ldots,S-1.$
As functions $f_k$ are transitive on $\Z/p^{k+1}\Z,$ then $F_{k,0}$ is transitive on $\Z/p\Z.$

Thus, $f_0=\varphi_0$ and $F_{k,0}$ is transitive on $\Z/p\Z.$ Then the function $f$ is ergodic by Theorem \ref{generg}.

Vice versa, let $f$ be an ergodic function. Then from Theorem \ref{generg}, in particular, follows that $f_{S-1}\equiv f \bmod p^S$ is transitive on $\Z/p^S\Z$
(particularly, the first condition of this Theorem is satisfied). And $F_{k,0}$ is transitive on $\Z/p\Z$ for $k=S,S+1,\ldots .$

As shown above, $F_{k,0}=Q_kx_k+L_k,$ $k=S,S+1,\ldots ,$ where $$Q_k=\left(\prod_{i=0}^{p^S-1} A_k(i)\right)^{p^{k-S}}$$ and $L_k$ are set as in the third condition of this Theorem.

From transitivity of $F_{k,0}=Q_kx_k+L_k$ follows that $Q_k=1,$ $L_k\neq 0.$ Indeed, suppose that $L_k=0.$ Then $F_{k,0}(0)=0,$ which contradict with transitivity of $F_{k,0}.$ Suppose that for $L_k\neq 0$ the value $Q_k\neq 1.$ Then the congruence $Q_kx_k+L_k=x_k \bmod p$ has solution, and therefore, the permutation $F_{k,0}$ has a fixed point.  This contradicts transitivity of $F_{k,0}.$

Relation $L_k\neq 0,$ $k=S,S+1,\ldots $ establishes the validity of the third condition of this Theorem. To verify the second condition we note that for $k=S,S+1,\ldots $
$$1=Q_k=\left(\prod_{i=0}^{p^S-1} A_k(i)\right)^{p^{k-S}}\equiv \prod_{i=0}^{p^S-1} A_k(i) \bmod p.$$
\end{proof}

Theorem \ref{ergfix} can be stated in the following equivalent form.

\begin{theorem}
\label{ergfixeq}
Let $p$-adic ($p\neq 2$) compatible and measure-preserving function $f\:\Z_p\>\Z_p$ be presented in the
 coordinate form (\ref{coordf}),
where $\varphi_k(x_0,x_1,\ldots, x_k)$ are such $p$-valued functions that 
$$\varphi_k(x_0,\ldots, x_k)=x_kA_k(x_0,\ldots,x_{S-1}) + \alpha_k(x_0,\ldots,x_{k-1})=x_kA_k(\bar{x}_{S-1})+\alpha_k(\bar{x}_{k-1}),$$
$k=S,S+1,\ldots$ for some fixed integer $S.$

The function $f$ is ergodic if and only if holds simultaneously
\begin{enumerate}
\item $f_S\equiv f \bmod p^{S+1}$ be transitive  (monocycle) permutation on the set  of residues modulo $p^{S+1};$
\item $\prod_{i=0}^{p^S-1} A_k(i)\equiv 1 \bmod p,$ $k=S+1,S+2,\ldots;$
\item for $k=S+1,S+2,\ldots$ 
\begin{align*}
&2^{-1} \sum_{i=0}^{p^S-1} \tau_k^{(i)}- \sum_{i=0}^{p^S-1} \frac{\tau_k^{(i)}}{p^k} [p^{k-S}(f^{(i+1)}(0) \bmod p^S)- \\
&-\sum_{\beta=0}^{p^{k-S}-1} f(f^{(i)}(0) \bmod p^S+p^S\beta) ]\neq 0 \bmod p,
\end{align*}
where $\tau_k^{(i)}=\prod_{j=i+1}^{p^S-1} A_k(f^{(j)}(0)),$ $i=0,1,\ldots,p^S-2$ and $\tau_k^{(i)}=1$ for $i=p^S-1.$
\end{enumerate}
\end{theorem}

\begin{proof}
For the proof of this Theorem it is sufficient to check  that the conditions of Theorem \ref{ergfix} are equivalent  to transitivity of 
$f_S\equiv f \bmod p^{S+1}$ modulo $p^{S+1}.$ This follows from the conditions for $k=S$ of Theorem \ref{ergfix}, see its proof.
Let $f_S\equiv f \bmod p^{S+1}$ is transitive on $\Z/p^{S+1}\Z.$ Then $f_{S-1}\equiv f \bmod p^S$ is transitive on $\Z/p^S\Z,$
and $F_{S,0}$ is transitive on $\Z/p\Z.$

As was shown in the proof of Theorem \ref{ergfix}, $F_{S,0}=Q_Sx_S+L_S,$ where $Q_S=\prod_{i=0}^{p^S-1} A_S(i)$
and
$$L_S= \sum_{i=0}^{p^S-1}  \frac{\tau_S^{(i)}}{p^S} \left[f(f^{(i)}(0) \bmod p^S )- f^{(i+1)}(0) \bmod p^S \right] \bmod p,$$
where
$\tau_S^{(i)}=\prod_{j=i+1}^{p^S-1} A_S(f^{(j)}(0)),$ $i=0,1,\ldots,p^S-2,$ and $\tau_S^{(i)}=1$ for $i=p^S-1.$
As $F_{S,0}$ is transitive on $\Z/p\Z,$ then $Q_S\equiv 1 \bmod p$ and $L_S\neq 0 \bmod p.$
\end{proof}

\begin{comment}
Theorems \ref{ergfix} and \ref{ergfixeq} describe ergodic functions such that their coordinate functions $\varphi_k(x_0,\ldots,x_k)$ are given  by linear (with respect to $x_k)$ polynomials, where coefficients for $x_k$ depend on no more than $S$ variables $x_0,\ldots,x_S$ for some fixed value  $S.$
Note that, in contrast to the general criterion of ergodicity (see Theorems \ref{generg} and \ref{generg2}), where it is required to calculate the values of all iterations $f^{(i)}(0) \bmod p^k,$ $i\in \left\{0,1,\ldots,p^k-1\right\},$ $k=1,2,\ldots ,$ too check  ergodicity of functions from the class represented in Theorems \ref{ergfix}, \ref{ergfixeq}  we need to calculate the iterations for the {\it limited number of values,} namely, $f^{(i)}(0) \bmod p^S,$ $i\in \left\{0,1,\ldots,p^S-1\right\},$ for every $k=S,S+1,\ldots .$
\end{comment}

Results  of Theorems \ref{ergfix} and \ref{ergfixeq} can be used to describe the  ergodic compatible $p$-adic functions, which additionally are uniformly  differentiable modulo $p.$
As defined in Defintion 3.27 and 3.28 [p.58, 60 \cite{ANKH}] the function $f\:\Z_p\>\Z_p$ is uniformly differentiable function modulo $p,$ if for any sufficiently large integer $k$ 
(i.e. for $k>N,$ where $N$-some fixed integer), comparison
$f(x+p^kh)\equiv f(x)+p^nh\cdot \partial f(x) (\bmod p^{k+1})$
is satisfied for any $x\in \Z_p$ and $h\in \Z_p.$ Such minimal $N$ we denote via $S.$

\begin{corollary}
\label{openq}
Let $f\:\Z_p\>\Z_p$ ($p\neq 2$) be a  compatible uniformly  differentiable modulo $p$ function.
The function $f$ is ergodic if and only if

\begin{enumerate}
\item $f_S\equiv f \bmod p^{S+1}$ be transitive  (monocycle) permutation on the set  of residues modulo $p^{S+1};$
\item $\prod_{i=0}^{p^S-1} \partial f(i)\equiv 1 \bmod p;$ 
\item for $k=S+1,S+2,\ldots$ 
\begin{align*}
&2^{-1} \sum_{i=0}^{p^S-1} \tau^{(i)}- \sum_{i=0}^{p^S-1} \frac{\tau^{(i)}}{p^k} [p^{k-S}(f^{(i+1)}(0) \bmod p^S)- \\
&-\sum_{\beta=0}^{p^{k-S}-1} f(f^{(i)}(0) \bmod p^S+p^S\beta) ]\neq 0 \bmod p,
\end{align*}
where $\tau^{(i)}=\prod_{j=i+1}^{p^S-1}\partial f(f^{(j)}(0)),$ $i=0,1,\ldots,p^S-2$ and $\tau^{(i)}=1$ for $i=p^S-1.$
\end{enumerate}
\end{corollary}

\begin{proof}
As the function $f$ is  compatible, then it can be presented in the coordinate form
$f(x_0+px_1+\ldots +p^kx_k+ \ldots)=\varphi_0(x_0)+p\varphi_1(x_0,x_1)+\ldots+p^k\varphi_k(x_0,x_1,\ldots,x_k)+\ldots,$
where $\varphi_k(x_0,x_1,\ldots, x_k)$ are $p$-valued functions. Then
$$\varphi_k(x_0,x_1,\ldots, x_k)=\frac{f(x_0+\ldots +p^kx_k)-f(x_0+\ldots +p^{k-1}x_{k-1}) (\bmod p^k)}{p^k} \bmod p.$$

As $f$ is uniformly  differentiable modulo $p$ function, then for $k>S$
\begin{align*}
&\frac{f(x_0+\ldots +p^kx_k)-f(x_0+\ldots +p^{k-1}x_{k-1})}{p^k}\equiv \\
&\equiv x_k\cdot \partial f(x_0+\ldots +p^{k-1}x_{k-1}) \bmod p.
\end{align*}

Let
\begin{align*}
&f(x_0+\ldots +p^{k-1}x_{k-1})\equiv f(x_0+\ldots +p^{k-1}x_{k-1}) \bmod p^k+ \\
&+p^k\alpha_k (x_0,x_1,\ldots, x_{k-1}) \bmod p^{k+1},
\end{align*}
where $\alpha_k (x_0,x_1,\ldots, x_{k-1})$ is $p$-valued function.

Then
$$\varphi_k(x_0,x_1,\ldots, x_k)= x_k\cdot \left(\partial f(x_0+\ldots +p^{k-1}x_{k-1}) \bmod p\right) + \alpha_k (x_0,\ldots, x_{k-1}).$$
By Proposition 3.35 [p. 63 \cite{ANKH}] the periodic function $\partial f(x_0+\ldots +p^{k-1}x_{k-1})$ has period $p^S,$ i.e.
$$\partial f(x_0+\ldots +p^{k-1}x_{k-1})\equiv \partial f(x_0+\ldots +p^{S-1}x_{S-1}) \bmod p^S,$$
and, therefore, for $k\geq S$
$$\partial f(x_0+\ldots +p^{k-1}x_{k-1})\equiv \partial f(x_0+\ldots +p^{S-1}x_{S-1}) \bmod p.$$

Denote $p$-valued function $\partial f(x_0+\ldots +p^{S-1}x_{S-1}) \bmod p,$ $k\geq S$ as \\
$A(x_0,x_1,\ldots, x_{S-1}).$ Thus, for $k\geq S,$ we obtain  
$$\varphi_k(x_0,x_1,\ldots, x_k)= x_k\cdot A(x_0,x_1,\ldots, x_{S-1}) + \alpha_k (x_0,x_1,\ldots, x_{k-1}).$$

This shows that compatible uniformly differentiable modulo $p$ functions  satisfy the conditions of Theorem \ref{ergfixeq} (as well as Theorem \ref{ergfix}).
To complete the proof it remains to use the criterion of ergodicity from Theorem \ref{ergfixeq}.
\end{proof}

\begin{comment}
At the Corollary \ref{openq} describes compatible uniformly differentiable modulo $p$ ergodic functions.
Thus, Open question 4.60 [p.132, \cite{ANKH}] are fully resolved. Note that compatible uniformly differentiable modulo $p^2$ ergodic functions are described by V. Anashin, see, for example, Theorem 4.55 [p.126 \cite{ANKH}].
\end{comment}

\section*{Acknowledgments}
 
This work is supported by the joint grant of Swedish and South-African Research Councils, ``Non-Archimedean analysis: from fundamentals 
to applications'' and the grant of the Faculty of natural Science and Engineering of Linnaeus University ``Mathematical Modeling of Complex Hierarchic Systems.''

\end{document}